\documentclass[12pt,reqno]{amsart}
\usepackage{amssymb,amsmath,amsthm,amsfonts,fancyhdr}
\usepackage{cases}
\usepackage{tikz}
\usepackage{indentfirst}
\usepackage{mathrsfs}
\usepackage{setspace}
\usepackage{color}
\usepackage{cite}
\usepackage{bm}
\usepackage{geometry}
\usepackage{enumitem}
\numberwithin{equation}{section}

\newtheorem{thm}{Theorem}[section]
\newtheorem{defn}[thm]{Defnition}
\newtheorem{lemma}[thm]{Lemma}
\newtheorem{cor}[thm]{Corollary}
\newtheorem{re}[thm]{Remark}

\renewcommand{\det}{\mbox{det}}

  \numberwithin{equation}{section}
  \numberwithin{figure}{section}

\begin{document}
\title[Liouville type theorems]{Liouville type theorems for fully nonlinear elliptic equations in exterior domains in half spaces}
\author{Dongsheng Li}
\author{Rulin Liu}
\address{School of Mathematics and Statistics, Xi'an Jiaotong University, Xi'an, P.R.China 710049.}
\address{School of Mathematics and Statistics, Xi'an Jiaotong University, Xi'an, P.R.China 710049.}
\email{lidsh@mail.xjtu.edu.cn}
\email{lrl001@stu.xjtu.edu.cn}
\begin{abstract}
  We establish a Liouville type theorem for fully nonlinear uniformly elliptic equations in exterior domains in half spaces under quadratic boundary data and a quadratic growth condition, that is, any viscosity solution tends to a quadratic polynomial plus lower order terms at infinity with the rate at least $|x|^{-\alpha-n}$ for any $\alpha\in(0,1)$.
  As applications, this result can lead to Liouville type theorems for Monge-Amp\`{e}re equations, $k$-Hessian equations and special Lagrangian equations with critical and supercritical phases.
\end{abstract}
\footnotetext{This research is supported by NSFC 12471202.}
\maketitle
\section{Introduction}

The classical Bernstein theorem for Monge-Amp\`{e}re equation $$\det{D^2u}=1$$ in $\mathbb{R}^n$,
which was proved by J\"{o}rgens \cite{JO}\ ($n=2$), Calabi\cite{CA}\ ($n\leq5$) and Pogolove\cite{PO}\ ($n\geq 2$), asserts that its classical convex solution must be a quadratic polynomial\ (without convexity hypothesis for $n=2$).
This result was extended to exterior domains in $\mathbb{R}^n$ by Ferrer, Mart\'{i}nez and Mil\'{a}n\cite{F-F-M} for $n=2$ and by Caffarelli and Li\cite{C-LI} for all $n\geq 2$, that is, any convex viscosity solution is asymptotic to a quadratic polynomial at infinity.
Recently, Han and Wang \cite{H-W} proved that the remainders in the asymptotic expansions can be characterized by a single function near the origin.
When the domain is half space $\mathbb{R}^n_+$, Savin\cite{SA2} proved a Liouville theorem for degenerate solutions to Monge–Ampere equations. Jia, the first author and Li\cite{JIA-LI-LI} showed that in exterior domains in $\mathbb{R}^n_+$, if the convex solution satisfies a quadratic boundary data and a quadratic growth condition, then it tends to a quadratic polynomial at infinity with rate at least $x_n/(|x|^n)$.

For another class of very important fully nonlinear elliptic equations, $k$-Hessian ($1< k<n$) equations $$\sigma_k\left(\lambda\left(D^2u(x)\right)\right)=1,$$ there are several Liouville type theorems. When $k=2$ ( i.e. quadratic Hessian equation ), Chang and Yuan \cite{CH-Y} proved that any global smooth convex solution in $\mathbb{R}^n$ must be quadratic if $D^2u(x)\geq\left(\delta-\sqrt{\frac 2{n(n-1)}}\right)I$ for any $\delta>0$. Later Shankar and Yuan \cite{Sh-Y1} improved this result to semiconvex solutions with $D^2u(x)>-KI$ for a large $K>0$.
Under quadratic growth assumption, Chen and Xiang \cite{CH-X} obtained that any entire 2-convex solution is a quadratic polynomial if $\sigma_3(\lambda)\geq -A$ for $A\geq 0$ (redundant for $n=3$).
For general $k$, Bao, Chen, Guan and Ji \cite{B-C-G-J} proved that a strictly convex smooth entire solution satisfying quadratic growth condition must be a quadratic polynomial.
The convexity of solutions in \cite{B-C-G-J} was relaxed to $(k+1)$-convexity by Li, Ren and Wang \cite{LI-R-W}. In \cite{Du}, Du proved three necessary and sufficient conditions for the condition of quadratic growth, which implies Liouville property to $k$-Hessian equations.
Recently, Chu and Dinew \cite{CH-D} established the Liouville theorem for a general class of Hessian equations and generalized the results of \cite{CH-X} and \cite{LI-R-W}.

There are also a few results for $k$-Hessian equations in half space $\mathbb{R}^n_+$. Zhou \cite{ZH} obtained the Liouville theorem for $k=2$ with convexity and quadratic growth hypotheses of solutions. Newly, Jia and Ma \cite{JIA-MA} proved that for all $1<k<n$ any $k$-convex solution in $\mathbb{R}^n_+$ with an appropriate boundary condition must be quadratic polynomial, provided a quadratic growth condition and a weak $(k+1)$-convexity.

We also consider special Lagrangian equation $$\sum\limits_{i=1}^n\arctan{\lambda_i\left(D^2u(x)\right)}=\Theta.$$ Its smooth convex entire solution must be quadratic was proved by Yuan \cite{Y1}.
Later, Yuan \cite{Y2} showed that any smooth entire solution with supercritical phases also must be quadratic.
This result was extended to exterior domain by the first author, Li and Yuan \cite{LI-LI-Y}, namely, for supercritical phases, a smooth solution in exterior domain in $\mathbb{R}^n$ is asymptotic to a quadratic polynomial for $n\geq 3$, with extra logarithmic term for $n=2$.
Recently, Han and Marchenko \cite{H-M} characterized the remainders in the asymptotic expansions by a single function near the origin.

Now let us focus on the general fully nonlinear elliptic equation
\begin{equation}\label{eq:F(D^2u)}
 F(D^2u(x))=0.
\end{equation}
The first author, Li and Yuan\cite{LI-LI-Y} established an exterior Liouville type theorem for uniformly elliptic equation (\ref{eq:F(D^2u)}) for dimension $n\geq 3$, which states that its smooth solution in $\mathbb{R}^n\setminus\overline{B_1}$ tends to a quadratic polynomial as $|x|\to\infty$ if $F$ is convex (or concave or the level set of $F$ is convex) and the $L^{\infty}$ norm of $D^2u$ is bounded.
Later, with bounded Hessian assumption but without concavity of $F$, the authors \cite{LI-LIU} proved that when dimension is $2$, the viscosity solution to (\ref{eq:F(D^2u)}) in exterior domains in $\mathbb{R}^2$ tends to a quadratic polynomial plus a logarithmic term and $x/(|x|^2)$ term as $|x|\to\infty$.
For half space $\mathbb{R}^n_+$, under some strong growth condition on solutions, the first author and Liang\cite{LI-LIANG} established a Liouville type theorem for equation (\ref{eq:F(D^2u)}) in exterior domains in $\mathbb{R}^n_+$.

In this paper, we establish a second order Liouville type theorem for general fully nonlinear uniformly elliptic equation (\ref{eq:F(D^2u)}) in exterior domains in half space $\mathbb{R}^n_+$ with the assumption that $u$ satisfies quadratic boundary data and quadratic growth condition, which can be regarded as the counterpart of the results in \cite{LI-LI-Y} and \cite{LI-LIU}.
Throughout this paper, we say a Liouville type theorem for some equation is $k$-th ($k\geq 0$) order if its solutions grow at most as $k$-th order polynomials at infinity. Our main theorem goes as following.
\begin{thm}\label{thm:main}
 Let $u\in C\left(\overline{\mathbb{R}^n_+}\setminus\overline{B_1}^+\right)$ be a viscosity solution to
 \begin{equation*}
  \begin{cases}
   F\left(D^2u(x)\right)=0\ &\text{in}\ \mathbb{R}^n_+\setminus\overline{B_1}^+,\\
   u(x)=p(x)\               &\text{on}\ \partial\mathbb{R}^n_+,
  \end{cases}
 \end{equation*}
 where $F\in C^{1,1}$ is a fully nonlinear uniformly elliptic operator with ellipticity constants $\lambda$ and $\Lambda$, $p$ is a quadratic polynomial.
 Assume that $u$ satisfies the quadratic growth condition, i.e.
 \begin{equation}\label{cond:qua growth of F}
  \left|u(x)\right|\leq M\left(1+|x|^2\right)
 \end{equation} in $\overline{\mathbb{R}^n_+}\setminus\overline{B_1}^+$ for some positive constant $M$.
 Then there exists a unique $A\in \mathcal{S}^n$, $b\in \mathbb{R}^n, c, d\in\mathbb{R}$ and some matrix $C\in \mathbb{R}^{n\times n}$ such that for any $\alpha\in(0,1)$,
 $$u(x)=\frac 12 x^{\mathrm{T}}Ax+b\cdot x+c+d\frac{x_n}{\left|x^{\mathrm{T}}Qx\right|^{\frac n2}}+\frac 12 \frac{x^{\mathrm{T}}Cx}{\left|x^{\mathrm{T}}Qx\right|^{\frac{n+2}2}}+O\left(|x|^{-\alpha-n}\right)\ \text{as}\ |x|\to\infty$$
 and $$\left.\left(\frac 12 x^{\mathrm{T}}Ax+b\cdot x+c\right)\right|_{\partial\mathbb{R}_+^n}=p(x),\ \left.\left(\frac 12 \frac{x^{\mathrm{T}}Cx}{\left|x^{\mathrm{T}}Qx\right|^{\frac{n+2}2}}+O\left(|x|^{-\alpha-n}\right)\right)\right|_{\partial\mathbb{R}_+^n}=0$$ if either $n\geq 3$ with assumption that $F$ is concave (or convex) or $n=2$, where $Q=\left[F_{ij}(A)\right]^{-1}$.

 Furthermore, if $F$ is smooth, then
 $$\left|D^k\left(u(x)-\frac 12 x^{\mathrm{T}}Ax-b\cdot x-c-d\frac{x_n}{\big|x^{\mathrm{T}}Qx\big|^{\frac n2}}-\frac 12 \frac{x^{\mathrm{T}}Cx}{\left|x^{\mathrm{T}}Qx\right|^{\frac{n+2}2}}\right)\right|=O\left(|x|^{-\alpha-n-k}\right)\ \text{as}\ |x|\to\infty$$
 for all $k\in\mathbb{N}$.
 \end{thm}
\begin{re}
 The concavity (or convexity) of $F$ for $n\geq 3$ in Theorem \ref{thm:main} can be relaxed to the convexity of the level set $\{N\in\mathcal{S}^n\mid F(N)=0\}$.
\end{re}
The following second order Liouville theorem for equation (\ref{eq:F(D^2u)}) in half space $\overline{\mathbb{R}^n_+}$ is a direct corollary of Theorem \ref{thm:main}.
\begin{cor}\label{cor:Liouville of F}
 Let $F$ and $p$ satisfy the hypotheses of Theorem \ref{thm:main}. Let $u\in C\left(\overline{\mathbb{R}^n_+}\right)$ be a viscosity solution to the Dirichlet problem
 \begin{numcases}{}
  $$F\left(D^2u(x)\right)=0$$\ &in\ $\mathbb{R}^n_+$,\notag\\
  $$u(x)=p(x)$$   \ &on\ $\partial\mathbb{R}^n_+$\notag
 \end{numcases}
 and satisfy quadratic growth condition (\ref{cond:qua growth of F}) in $\overline{\mathbb{R}^n_+}$ for some positive constant $M$.
 Then $u$ must be a quadratic polynomial.
\end{cor}

In order to prove the second order Liouville type theorem, i.e. Theorem \ref{thm:main}, the first step is to determine the limit of the Hessian $D^2u(x)$ at infinity and the rate at which  $D^2u(x)$ tends to its limit as $|x|$ tends to infinity.
In view of the Dirichlet boundary data, we can take derivatives to equation (\ref{eq:F(D^2u)}) with respect to $x_k$ for $k=1,2,\cdots,n-1$ to obtain a linear uniformly elliptic equation.
Since $u_{x_k}$ solves a Dirichlet problem of the linearized equation with zero boundary data after normalizing and $u_{x_k}$ grows linearly, we only need to show the existence of limits of the gradient for solutions to the linearized equation and estimate the rate of decay at infinity.
Then the limits and the decay rate of second order derivatives of $u$ (except for $u_{nn}$)  at infinity will be obtained. As for the behavior of $u_{nn}$ at infinity, we determine it by making use of equation.

To do this, we establish a first order Liouville type theorem for linear uniformly elliptic equations in $\mathbb{R}^n_+$ and extend this result to exterior domains subsequently. Then the gradient asymptotic behavior we desired at infinity follows.
This method works for all dimensions $n\geq 2$ and there is an alternative approach for dimension 2, i.e. we obtain the gradient H\"{o}lder continuity and asymptotic behavior for solutions to the linearized equation by showing the H\"{o}lder continuity and asymptotic behavior at infinity of exterior $k$-quasiconformal mappings in $\mathbb{R}^2_+\setminus\overline{B_1}^+$.

After we finish the first step, the remaining linear approach which determines the linear term in the asymptotic expansion is just a standard way (cf. \cite{JIA-LI-LI}) and we absorb it. The finer asymptotic behavior can be obtained by Kelvin transform and Schauder estimate.

The organization of this paper goes as follows. In Section 2, we establish first order Liouville type theorems for linear uniformly elliptic equations in $\overline{\mathbb{R}^n_+}$ and $\overline{\mathbb{R}^n_+}\setminus \overline{B_1}^+$ respectively.
In Section 3, we finish the proof of Theorem \ref{thm:main}. The H\"{o}lder continuity and asymptotic behavior at infinity of exterior $K$-quasiconformal mappings will be studied in Section 4, which leads to the asymptotic behavior of Hessian of $u$ at infinity for dimension 2.
In the last section, we apply Theorem \ref{thm:main} to Monge-Amp\`{e}re equations, $k$-Hessian equations and special Lagrangian equations with critical and supercritical phases to obtain Liouville type theorems in exterior domains in half spaces respectively.

The following notations will be used throughout this paper.
\begin{itemize}[leftmargin=10pt]
  \item For any $x\in\mathbb{R}^n$,\ $x=(x^{\prime},x_n),\  x^{\prime}\in\mathbb{R}^{n-1}.$
  \item $\mathbb{R}^n_+=\{x\in\mathbb{R}^n\mid x_n>0\},\
        \overline{\mathbb{R}^n_+}=\{x\in\mathbb{R}^n\mid x_n\geq 0\}.$
  \item $B_r(x_0)=\left\{x\in\mathbb{R}^n\mid|x-x_0|<r\right\}$,\  $B_r^+(x_0)=B_r(x_0)\cap\mathbb{R}^n_+$,\ $\overline{B_r(x_0)}^+=\overline{B_r(x_0)}\cap\mathbb{R}^n_+$,\ $\overline{B_r^+(x_0)}=\overline{B_r(x_0)}\cap\overline{\mathbb{R}^n_+}$,\  $B_r=B_r(0)$\ for any\ $x_0\in\mathbb{R}^n$\ and\ $r>0$.
  \item $Q_r^+=\left\{(x^{\prime},x_n)\in\mathbb{R}^n_+\mid|x^{\prime}|<r, 0<x_n<r\right\}$\ for any $r>0$.
  \item $I$ denotes the identity matrix in $\mathbb{R}^{n\times n}$.
  \item $\mathcal{S}^n$ denotes the space of all symmetric matrices in $\mathbb{R}^{n\times n}$.
  \item $F_{ij}(N)=\frac{\partial}{\partial N_{ij}}F(N)$\ for any $N\in\mathcal{S}^n$.
\end{itemize}
\section{first order Liouville type theorems for linear equations in half spaces}
As mentioned in Section 1, the key step to determine the behavior of the Hessian of solutions to equaiton (\ref{eq:F(D^2u)}) at infinity is to study the gradient behavior of solutions to the linearized equation.
In this section, we study the first order Liouville type theorems for linear uniformly elliptic equations in half space $\overline{\mathbb{R}^n_+}$ and exterior domain $\overline{\mathbb{R}^n_+}\setminus\overline{B_1}^+$, i.e. Theorem \ref{thm:Liouville linear half spaces} and Theorem \ref{thm:exterior Liouville linear half spaces} respectively.
\begin{thm}
 \label{thm:Liouville linear half spaces}
 Let $a_{ij}\in C\left(\overline{\mathbb{R}^n_+}\right)$ such that $\lambda I\leq\left[a_{ij}(x)\right]\leq\Lambda I$ in $\overline{\mathbb{R}^n_+}$ for some $0<\lambda\leq\Lambda<\infty$ and $u\in C\left(\overline{\mathbb{R}^n_+}\right)$ be a solution to
 \begin{numcases}{}
  $$a_{ij}(x)u_{ij}(x)=0$$\ &in\ $\mathbb{R}^n_+$,\notag\\
  $$u(x)=0              $$\ &on\ $\partial\mathbb{R}^n_+$.\notag
 \end{numcases}
 Assume that $u$ satisfies the linear growth condition, i.e.
 \begin{equation}
  \left|u(x)\right|\leq \bar{C}(1+\left|x\right|)\label{cond:linear growth}
 \end{equation}
 in $\overline{\mathbb{R}^n_+}$ for some positive constant $\bar{C}$. Then there exists some constant $a^{\star}$ depending on $n,\lambda,\Lambda$ and $\bar{C}$ such that
 $$u(x)=a^{\star} x_n\ \text{in}\ \overline{\mathbb{R}^n_+}.$$
\end{thm}
\begin{proof}
 Let us begin by the argument of barrier. Assume $\bar{C}=1$, otherwise we divide by $\bar{C}$ from $u$. Suppose $R>2$. Define
 \begin{equation}\label{eq:barrier}
  \psi_R(x):=2\left[\left((n-1)\frac{\Lambda}{\lambda}+1\right)x_n-(n-1)\frac{\Lambda}{\lambda}\frac 1R x_n^2+\frac 1R\left|x^{\prime}\right|^2\right],\ x\in\overline{Q_R^+}.
 \end{equation}
 Then $\psi_R$ is twice differentiable and $\psi_R$ is an upper barrier to $u$ on $\overline{Q_R^+}$.

 Indeed, on $\partial Q_R^+\cap\{x_n=0\},$ $$\psi_R(x)=\frac 2R|x^{\prime}|^2\geq u(x)\ \text{with}\ "="\ \text{holds at}\ 0.$$
 On $\partial Q_R^+\cap\{x_n=R\}$, $$\psi_R(x)\geq 2\left[\left((n-1)\frac{\Lambda}{\lambda}+1\right)x_n-(n-1)\frac{\Lambda}{\lambda}\frac 1R x_n^2\right]=2R\geq u(x).$$
 On $\partial Q_R^+\cap\{|x^{\prime}|=R\}$, since $\left((n-1)\frac{\Lambda}{\lambda}+1\right)x_n-(n-1)\frac{\Lambda}{\lambda}\frac 1R x_n^2\geq 0\ \text{for}\ x_n\in(0,R)$, $$\psi_R(x)\geq \frac 2R|x^{\prime}|^2=2R\geq u(x).$$
 Moreover, for $x\in Q_R^+$, $$a_{ij}(x)\left(\psi_R\right)_{ij}(x)=2\left(\sum_{i=1}^{n-1}a_{ii}\frac 2R-a_{nn}(n-1)\frac{\Lambda}{\lambda}\frac 2R\right)\leq \frac 4R\left(\sum_{i=1}^{n-1}\Lambda-\lambda(n-1)\frac{\Lambda}{\lambda}\right)=0.$$
 Therefore, $u(x)\leq \psi_R(x)$ on $\partial Q_R^+$ and $a_{ij}(x)\left(\psi_R\right)_{ij}(x)\leq 0$ in $Q_R^+$. We conclude by comparison principle that $$u(x)\leq \psi_R(x)\ \text{on}\ \overline{Q_R^+}.$$

 For any fixed $x_0\in\mathbb{R}^n_+$, we can take $R>2$ large enough so that $x_0\in Q_R^+$, then $$u(x_0)\leq \psi_R(x_0).$$
 Letting $R\to +\infty$ and by the arbitrarity of $x_0$, we obtain $$u(x)\leq 2\left((n-1)\frac{\Lambda}{\lambda}+1\right)x_n\ \text{in}\ \overline{\mathbb{R}^n_+}.$$

 Define $$a^{\star}:=\inf \left\{a\in\mathbb{R}\ |\ u(x)\leq ax_n\ \text{in}\ \overline{\mathbb{R}^n_+}\right\}.$$ Clearly, the set in the right hand side is not empty. Thus $$u(x)\leq a^{\star} x_n\ \text{in}\ \overline{\mathbb{R}^n_+}.$$
 We claim $$u(x)= a^{\star} x_n\ \text{in}\ \overline{\mathbb{R}^n_+}.$$
 Let $$v(x):= a^{\star} x_n-u(x).$$  Then $v(x)\geq 0$ in $\mathbb{R}^n_+$ and $v(x)$ vanishes continuously on $\partial\mathbb{R}^n_+$.
 We consider ${v(x)}/{x_n}$ at a point $Re_n$ with $R>0$ and $e_n=(0,\cdots,0,1)$ and we define $$V:= \liminf_{R\to\infty}\frac{v(Re_n)}{R}.$$
 By the linear growth condition (\ref{cond:linear growth}),
 $$0\leq\frac{v(Re_n)}{R}=\frac{a^{\star}R-u(Re_n)}{R}\leq\frac{a^{\star} R+(1+R)}{R}=a^{\star}+1+\frac{1}{R},$$ then we know $0\leq V<+\infty$.

 If $V>0$, then for any $0<\varepsilon<V$, there exists $R_0>0$ such that for all $R>R_0$, $v(Re_n)/R\geq V-\varepsilon$.
 It follows from the boundary Harnack principle that there exists some positive constant $C$ depending on $n, \lambda$ and $\Lambda$, such that for any fixed $x=(x^{\prime},x_n)\in Q_R^+$, $$\frac{v(x)}{x_n}\geq C\frac{v(Re_n)}{R}\geq C(V-\varepsilon),$$
 which implies $$u(x)\leq\left(a^{\star}-C(V-\varepsilon)\right)x_n.$$ By the arbitrarity of $R$, $$u(x)\leq(a^{\star}-C(V-\varepsilon))x_n\ \text{for any}\ x\in\mathbb{R}^n_+,$$ which contradicts to the definition of $a^{\star}$.

 Thus $V=0$ and for any $\varepsilon>0$, there exists a sequence $\left\{R_k\right\}_{k=1}^{\infty}$ depending on $\varepsilon$ with $R_k\to\infty$ as $k\to\infty$ such that $v(R_ke_n)/R_k\leq\varepsilon$. Again, by the boundary Harnack principle, we see that there exists a positive constant $C$ depending on $n, \lambda$ and $\Lambda$ such that for any fixed $x\in Q_{R_k}^+$,
 $$\frac{v(x)}{x_n}\leq C\frac{v\left(R_ke_n\right)}{R_k}\leq C\varepsilon.$$
 Accordingly, $$u(x)\geq(a^{\star}-C\varepsilon)x_n.$$ Taking $\varepsilon\to 0$, $$u(x)\geq a^{\star}x_n\ \text{for any}\ x\in\mathbb{R}^n_+.$$ Hence $$u(x)=a^{\star}x_n,\ x\in\mathbb{R}^n_+.$$

 We have proved the claim and the proof of this theorem can be closed.
\end{proof}
Next, we extend Theorem \ref{thm:Liouville linear half spaces} to exterior domains.
\begin{thm}\label{thm:exterior Liouville linear half spaces}
 Let $a_{ij}\in C\left(\overline{\mathbb{R}^n_+}\setminus \overline{B_1}^+\right)$ such that $\lambda I\leq\left[a_{ij}(x)\right]\leq\Lambda I$ for some $0<\lambda\leq\Lambda<\infty$ and $u\in C\left(\overline{\mathbb{R}^n_+}\setminus\overline{B_1}^+\right)$ be a solution to
 \begin{equation}\label{eq:linear ext domain}
  \begin{cases}
  a_{ij}(x)u_{ij}(x)=0\ &\text{in}\ \mathbb{R}_+^n\setminus\overline{B_1}^+, \\
  u(x)=0              \ &\text{on}\ \partial\mathbb{R}_+^n.
  \end{cases}
 \end{equation}
 Assume that $u$ satisfies the linear growth condition (\ref{cond:linear growth}) in  $\overline{\mathbb{R}_+^n}\setminus\overline{B_1}^+$ for some positive constant $\bar{C}$. Then $u\in C^{1,\varepsilon}\left(\overline{\mathbb{R}_+^n}\setminus\overline{B_1^+}\right)$ for any $\varepsilon\in(0,1)$ and there exists some constant $a^{\star}$ such that
 \begin{equation}\label{in:bound of u-a*x_n}
  \left|u(x)-a^{\star}x_n\right|\leq C\ \text{in}\ \overline{\mathbb{R}^n_+}\setminus\overline{B_1^+}
 \end{equation} and for some $b=(0,\cdots,0,a^{\star})^{\mathrm{T}}\in\mathbb{R}^n$,
 \begin{equation}\label{in:decay of u-a*x_n}
  \left|Du(x)-b\right|\leq C|x|^{-1}\ \text{in}\ \overline{\mathbb{R}^n_+}\setminus \overline{B_1^+},
 \end{equation} where $a^{\star}$ and $C$ depend on $n,\lambda,\Lambda$ and $\bar{C}$.
\end{thm}
\begin{proof}
 Assume $\bar{C}=1$, otherwise we divide by $\bar{C}$ from $u$. Since $a_{ij}\in C\left(\overline{\mathbb{R}^n_+}\setminus \overline{B_1}^+\right)$, we have, by the regularity theory, that $u\in C^{1,\varepsilon}\left(\overline{\mathbb{R}_+^n}\setminus\overline{B_1^+}\right)$ for any $\varepsilon\in(0,1)$. Next we show (\ref{in:bound of u-a*x_n}) holds.

 Let $\left\{R_k\right\}_{k=1}^{\infty}$ be a monotonic sequence such that $R_k>2$ and $R_k\to\infty$ as $k\to\infty$ and let $\left\{u_k\right\}_{k=1}^{\infty}$ be a sequence of solutions to the Dirichlet problem
 \begin{numcases}{}
  $$a_{ij}(x)\left(u_k\right)_{ij}(x)=0$$\ &\textup{in}\ $Q_{R_k}^+$,\label{eq:linear in Q_{R_k}}\\
  $$u_k(x)=u(x)$$                        \ &\textup{on}\ $\partial Q_{R_k}^+$.\notag
 \end{numcases}
 Thus $$-(1+|x|)\leq u_k(x)\leq 1+|x|\ \text{on}\ \partial Q_{R_k}^+.$$
 Let $\psi_{R_k}(x)$ be defined as (\ref{eq:barrier}) with respect to $R_k$ in $Q_{R_k}^+$. The arguments in the proof of Theorem \ref{thm:Liouville linear half spaces} implies $$-\psi_{R_k}(x)\leq u_k(x)\leq\psi_{R_k}(x)\ \text{on}\ \overline{Q_{R_k}^+}.$$

 For any fixed point $x_0=\left((x_0)^{\prime}, (x_0)_n\right)\in\mathbb{R}_+^n$, there exists some $\tilde{k}$ such that $x_0\in Q_{R_k}^+$ for all $k\geq \tilde{k}$ and the $C^{1,\varepsilon}$ regularity asserts that there exists a subsequence $\left\{u_{k_l}\right\}_{l=1}^{\infty}\subset\left\{u_k\right\}_{k=\tilde{k}}^{\infty}$ which converges uniformly with its first derivatives in any compact subset of $Q_{R_{\tilde{k}}}^+$ to a solution $u_{\infty}$ in $Q_{R_{\tilde{k}}}^+$ of equation (\ref{eq:linear in Q_{R_k}}).
 Hence $u_{\infty}\in C^{1,\varepsilon}\left(Q_{R_{\tilde{k}}}^+\cup\left\{|x^{\prime}|<R_{\tilde{k}},x_n=0\right\}\right)$ for any $\varepsilon\in(0,1)$.
 In view of $$-\psi_{R_{k_l}}(x_0)\leq u_{k_l}(x_0)\leq\psi_{R_{k_l}}(x_0)\ \text{on}\ \overline{Q_{R_{k_l}}^+},$$ we see that $$-2\left((n-1)\frac{\Lambda}{\lambda}+1\right)(x_0)_n\leq u_{\infty}(x_0)\leq 2\left((n-1)\frac{\Lambda}{\lambda}+1\right)(x_0)_n$$ by taking $l\to\infty$.
 By the arbitrarity of $x_0$, we have $$-2\left((n-1)\frac{\Lambda}{\lambda}+1\right)x_n\leq u_{\infty}(x)\leq 2\left((n-1)\frac{\Lambda}{\lambda}+1\right)x_n\ \text{in}\ \overline{\mathbb{R}_+^n}.$$
 Since it is easy to see that $u_{\infty}$ satisfies
 \begin{equation*}
  \begin{cases}
   a_{ij}(x)(u_{\infty})_{ij}(x)=0\ &\text{in}\ \mathbb{R}_+^n,\\
   u_{\infty}(x)=0               \ &\text{on}\ \partial\mathbb{R}_+^n,
  \end{cases}
 \end{equation*}
 it follows from Theorem \ref{thm:Liouville linear half spaces} that $$u_{\infty}(x)= a^{\star}x_n\ \text{in}\ \overline{\mathbb{R}^n_+}$$ for some constant $a^{\star}$ depending on $n,\lambda$ and $\Lambda$.

 On the other hand, since $|u(x)|$ and $|u_{k_l}(x)|$ are bounded on $\left(\partial B_1\right)^+$, we know that for some nonnegative constant $C$ depending on $n, \lambda$ and $\Lambda$,
 $$u_{k_l}(x)-C\leq u(x)\leq u_{k_l}(x)+C\ \text{on}\  \partial\left(Q_{R_{k_l}}^+\setminus\overline{B_1^+}\right)$$
 and comparison principle leads to $$u_{k_l}(x)-C\leq u(x)\leq u_{k_l}(x)+C\ \text{on}\ \overline{Q_{R_{k_l}}^+\setminus\overline{B_1^+}}.$$
 Noticing that $u_{k_l}$ converges to $u_{\infty}$ pointwisely, we obtain $$u_{\infty}(x)-C\leq u(x)\leq u_{\infty}(x)+C,\ x\in\mathbb{R}^n_+\setminus\overline{B_1^+},$$
 namely, $$\left|u(x)-a^{\star}x_n\right|\leq C,\ x\in\overline{\mathbb{R}^n_+}\setminus\overline{B_1^+}.$$

 Now we show (\ref{in:decay of u-a*x_n}) holds to end the proof. Since $$v(x):=u(x)-a^{\star}x_n$$ satisfies Dirichlet problem (\ref{eq:linear ext domain}), we know that $v\in C^{1,\varepsilon}\left(\overline{\mathbb{R}^n_+}\setminus\overline{B_1^+}\right)$ for any $\varepsilon\in(0,1)$.
 Then $$\left|v(x)\right|\leq C,\ x\in\overline{\mathbb{R}^n_+}\setminus\overline{B_1^+}$$
 implies $$\left|Dv(x)\right|\leq C|x|^{-1},\ x\in\overline{\mathbb{R}^n_+}\setminus\overline{B_1^+}.$$
 \end{proof}

\section{Second order Liouville type theorems in half spaces}
In this section, we prove the second order Liouville type theorem in exterior domain in half space, i.e. Theorem \ref{thm:main}. For this purpose, we begin with determining the limit $A$ of the Hessian $D^2u(x)$ at infinity and estimating the decay rate of $|D^2u(x)-A|$ as $|x|\to\infty$.
\begin{thm}
 \label{thm:limit of Hessian}
Let $u$ be as in Theorem \ref{thm:main}. Then there exists a unique $A\in\mathcal{S}^n$ such that
 $$D^2u(x)\to A\ \text{as}\ |x|\to\infty$$ and $$|D^2u(x)-A|\leq C|x|^{-1}\ \text{as}\ |x|\to\infty,$$ which implies
 \begin{eqnarray}\label{in:u-xAx}
  \left|u(x)-\frac 12 x^{\mathrm{T}}Ax\right|\leq C|x|\ \text{as}\ |x|\to\infty
 \end{eqnarray}
 and $$\left.\frac 12 x^{\mathrm{T}}Ax\right|_{\partial\mathbb{R}_+^n}=p(x),$$ where $C>0$ depends on $n,\lambda,\Lambda$ and $M$.
\end{thm}
\begin{proof}
 To simplify the proof, we make proper normalization. We assume, by subtracting an affine function, that $p(x)=\frac 12 x^{\mathrm{T}}Px$ for some $P\in\mathcal{S}^n$ with $P=\begin{pmatrix} \tilde{P} & b \\ b^{\mathrm{T}} & c \end{pmatrix}$ for $\tilde{P}\in\mathcal{S}^{n-1}, b\in\mathbb{R}^{n-1}$ and $c\in\mathbb{R}$.
 Then there exists an orthogonal matrix $\tilde{D}\in\mathbb{R}^{(n-1)\times(n-1)}$ such that $\tilde{D}^{\mathrm{T}}\tilde{P}\tilde{D}=I_{n-1}$.
 It is easy to check that $D=\begin{pmatrix} \tilde{D} & 0 \\ 0 & 1 \end{pmatrix}$ is orthogonal and
 $D^{\mathrm{T}}PD=\begin{pmatrix} I_{n-1} & \tilde{D}^{\mathrm{T}}b \\ b^{\mathrm{T}}\tilde{D} & c \end{pmatrix}$.
 Therefore we can assume that $$p(x)=\frac 12 |x^{\prime}|^2\ \text{on}\ \partial\mathbb{R}^n_+.$$

 Since $F\in C^{1,1}$ and $u(x)=\frac 12|x^{\prime}|^2$ on $\partial\mathbb{R}^n_+$, by the Nirenberg estimate (for $n=2$) or Evans-Krylov estimate (for $n\geq 3$), we can obtain $u\in C^{2,\tau}\left(\overline{\mathbb{R}^n_+}\setminus \overline{B_1^+}\right)$ for some $\tau\in(0,1)$ depending on $n, \lambda$ and $\Lambda$.
 Then we can take derivatives with respect to $x_k$ for $k=1,2,\cdots,n-1$ on both sides of equation (\ref{eq:F(D^2u)}) and subtract $x_k$ from $u_{x_k}$ to define $$v(x):=u_{x_k}(x)-x_k.$$ Thus $v$ satisfies
 \begin{equation}\label{eq:linearized F}
  \begin{cases}
   a_{ij}(x)v_{ij}(x)=0\ &\text{in}\ \mathbb{R}^n_+\setminus\overline{B_1}^+,\\
   v(x)=0               \ &\text{on}\ \partial\mathbb{R}^n_+,
  \end{cases}
 \end{equation} where $a_{ij}(x)=F_{ij}\left(D^2u(x)\right)$ is uniformly elliptic and $a_{ij}\in C^{\tau}\left(\mathbb{R}^n_+\setminus\overline{B_1}^+\right)$.
 By Schauder estimate, $v\in C^{2,\tau}\left(\overline{\mathbb{R}^n_+}\setminus \overline{B_1^+}\right)$, consequently $u\in C^{3,\tau}\left(\overline{\mathbb{R}^n_+}\setminus \overline{B_1^+}\right)$.
 The quadratic growth hypothesis (\ref{cond:qua growth of F}) implies the linear growth of $v$, i.e. for all $x\in\overline{\mathbb{R}^n_+}\setminus\overline{B_1^+}$, \
 \begin{equation}\label{cond:linear growth of v} |v(x)|\leq \bar{C}(1+|x|), \end{equation}
 where $\bar{C}$ is positive and depends only on $M$.

 According to (\ref{eq:linearized F}) and (\ref{cond:linear growth of v}), Theorem \ref{thm:exterior Liouville linear half spaces} assures that there exists some constant $a^{\star}$ depending on $n, \lambda, \Lambda$ and $M$ such that for some $b=(0,\cdots,0,a^{\star})^{\mathrm{T}}\in\mathbb{R}^n$,
 $$\left|Dv(x)-b\right|\leq \tilde{C}|x|^{-1},\  x\in\overline{\mathbb{R}^n_+}\setminus\overline{B_1^+},$$ where $\tilde{C}$ depends on $n, \lambda, \Lambda$ and $M$.
 Recall $v(x)=u_{x_k}(x)-x_k$ with $k=1,2,\cdots,n-1$, then we have $$\left|u_{kj}(x)-\delta_{kj}\right|\leq \tilde{C}|x|^{-1}\ \text{for}\ j=1,2,\cdots,n-1\ \text{and}\ \left|u_{kn}(x)-a^{\star}\right|\leq \tilde{C}|x|^{-1},\ x\in\overline{\mathbb{R}^n_+}\setminus\overline{B_1^+}.$$

 Only the limit of $u_{nn}(x)$ at infinity remains. For some $R>2$, set $\Omega_R=B_R^+\big\backslash\overline{B_{R/2}^+}$. Since for each $x,y\in\Omega_R$, $$F(D^2u(x))=0\ \text{and}\ F(D^2u(y))=0,$$ we can obtain
 $$\bar{a}_{ij}(u_{ij}(x)-u_{ij}(y))=0,\ x,y\in B_R^+\big\backslash\overline{B_{R/2}^+},$$
 where $\bar{a}_{ij}=\int_0^1F_{ij}\left(tD^2u(x)+(1-t)D^2u(y)\right) \mathrm{d}t$ is uniformly elliptic with the ellipticity constants depend on $\lambda$ and $\Lambda$.
 Hence $$\left|u_{nn}(x)-u_{nn}(y)\right|\leq \hat{C}\sum\limits_{i+j\leq 2n-1}\left|u_{ij}(x)-u_{ij}(y)\right|,\ x,y\in\Omega_R,$$ where $\hat{C}$ depends on $n, \lambda$ and $\Lambda$.
 This implies $$\mathop{\mathrm{osc}}\limits_{\Omega_R}\ u_{nn}\leq \hat{C}\sum\limits_{i+j\leq 2n-1}\mathop{\mathrm{osc}}\limits_{\Omega_R}\ u_{ij}\to 0\ \text{as}\ R\to\infty.$$
 In view of the continuity of $u_{nn}$ up to $\partial\mathbb{R}^n_+$, we get that $u_{nn}(x)$ tends to a limit $u_{nn}(\infty)$ at infinity.
 On the other hand, by taking $|y|\to\infty$, we have $$\left|u_{nn}(x)-u_{nn}(\infty)\right|\leq \hat{C}\sum\limits_{i+j\leq 2n-1}\left|u_{ij}(x)-u_{ij}(\infty)\right|\leq (n^2-1)\hat{C}\tilde{C}|x|^{-1}\ \text{as}\ |x|\to\infty.$$

 Therefore, there exists $A=\begin{pmatrix} I_{n-1} & a \\ a^{\mathrm{T}} & u_{nn}(\infty) \end{pmatrix}$ with $a=( a^{\star},\cdots, a^{\star})^{\mathrm{T}}\in\mathbb{R}^{n-1}$ such that
 $$D^2u(x)\to A\ \text{as}\ |x|\to\infty$$ and $$\left|D^2u(x)-A\right|\leq C|x|^{-1}\ \text{as}\ |x|\to\infty,$$ where $C$ depends on $n, \lambda, \Lambda$ and $M$. This implies
 $$\left|u(x)-\frac 12 x^{\mathrm{T}}Ax\right|\leq C|x|\ \text{as}\ |x|\to\infty$$ and $$\left.\frac 12 x^{\mathrm{T}}Ax\right|_{\partial\mathbb{R}_+^n}=\frac 12|x^{\prime}|^2.$$
\end{proof}
Once we found the limit $A$ of the Hessian $D^2u(x)$ and determined decay rate of $\left|D^2u(x)-A\right|$ as $|x|\to\infty$, the linear approach to asymptotic expasion of $u(x)$ at infinity can be finished by a standard way (cf. \cite{JIA-LI-LI}) and the finer asymptotic behavior follows from Kelvin transform and Schauder estimate, then Theorem \ref{thm:main} can be completed.

Before showing the details, we prove the following well known lemma, which describes the decay of derivatives.
\begin{lemma}
 \label{le:higher order decay}
 Let $\phi(x)\in C\left(\overline{\mathbb{R}^n_+}\setminus \overline{B_1^+}\right)$ be a viscosity solution to
 \begin{equation*}
  \begin{cases}
  F\left(D^2\phi(x)+A\right)=0\ &\text{in}\ \mathbb{R}^n_+\setminus \overline{B_1^+},\\
  \phi(x)=0                   \ &\text{on}\ \partial\mathbb{R}^n_+,
  \end{cases}
 \end{equation*}
 where $F\in C^{1,1}$ is a fully nonlinear uniformly elliptic operator with ellipticity constants $\lambda$ and $\Lambda$, $A\in\mathcal{S}^n$. Assume that for some constants $C_0>0$ and $\rho<2$,
 $$|\phi(x)|\leq C_0|x|^{\rho},\ x\in\overline{\mathbb{R}^n_+}\setminus \overline{B_1^+}.$$
 Then there exists some constant $r\geq 1$ depending on $n, C_0$ and $\rho$ such that if either $n\geq 3$ with the assumption that $F$ is concave or $n=2$, we have, for some $\tau\in(0,1)$ depending on $n, \lambda$ and $\Lambda$, that $\phi(x)\in C^{3,\tau}\left(\overline{\mathbb{R}^n_+}\setminus\overline{B_r^+}\right)$ and
 $$\left|D^k\phi(x)\right|\leq C|x|^{\rho-k},\ x\in\overline{\mathbb{R}^n_+}\setminus \overline{B_r^+},$$ where $k=0,1,2,3$, $C$ depends on $n$, $\lambda$, $\Lambda$, $C_0$, $\rho$ and the $C^{1,1}$ norm of $F$.
\end{lemma}
\begin{proof}
 Let $x_0$ be any point in $\overline{\mathbb{R}^n_+}$ with $|x_0|=R>2$. Define $\Omega_i:=\left\{y\in B_i\mid x_0+\frac R4 y\in\mathbb{R}^n_+\right\}$ and $$\phi_R(y):=\left(\frac 4R\right)^2\phi\left(x_0+\frac R4 y\right),\ y\in\Omega_1.$$
 Then there exists some $r\geq 1$ depending on $n, C_0$ and $\rho$ such that for all $x_0$ with $|x_0|\geq r$, $$\left|\phi_R(y)\right|\leq C_0\left(\frac 4R\right)^2\left|x_0+\frac R4 y\right|^{\rho}\leq 1,\ y\in\Omega_1,$$ and $\phi_R(y)$ satisfies
 \begin{equation*}
  \begin{cases}
  F\left(D^2\phi_R(y)+A\right)=0\ &\text{in}\ \Omega_1,\\
  \phi_R(y)=0                   \ &\text{on}\ \partial\Omega_1\cap\partial\mathbb{R}^n_+.
  \end{cases}
 \end{equation*}
 By the Nirenberg estimate (for $n=2$) or Evans-Krylov estimate (for $n\geq 3$) and Schauder estimates, we get that for some $\tau\in(0,1)$, $\phi_R\in C^{3,\tau}\left(\Omega_{1/2}\right)$ and $$\left|D^k\phi_R(0)\right|\leq C_1\|\phi_R\|_{L^{\infty}(\Omega_1)}\leq C|x_0|^{\rho-2},\ k=0,1,2,3,$$
 where $\tau$ depends on $n, \lambda$ and $\Lambda$, $C_1$ depends on $n$, $\lambda$, $\Lambda$, $\rho$ and the $C^{1,1}$ norm of $F$, $C$ depends on $n$, $\lambda$, $\Lambda$, $C_0$, $\rho$ and the $C^{1,1}$ norm of $F$. Hence for all $x_0$ with $|x_0|\geq r$, $$\left|D^k\phi(x_0)\right|\leq C|x_0|^{\rho-k}$$ and the proof can be closed by the arbitrarity of $x_0$.
\end{proof}
Now we complete the proof of Theorem \ref{thm:main}
\begin{proof}[Proof of Theorem \ref{thm:main}.]
 Let $$\varphi(x)=u(x)-\frac 12 x^{\mathrm{T}}Ax.$$ We assume, without loss of generality, that $$F_{ij}(A)=\delta_{ij}.$$ In fact, there exists a invertible matrix $Q\in\mathbb{R}^{n\times n}$ such that $[F_{ij}(A)]=(Q^{-1})^{\mathrm{T}}(Q^{-1})$ since $[F_{ij}(A)]$ is positive definite and the assumption can be done after the coordinate transformation $\tilde{x}=Qx$. There are three steps to finish the proof.

 \textit{Step 1. Determining the linear term.}

 Based on Theorem \ref{thm:limit of Hessian} and Lemma \ref{le:higher order decay}, we have that $\varphi(x)$ satisfies
 \begin{numcases}{}
  $$F(D^2\varphi(x)+A)=0$$\ &in\ $\mathbb{R}^n_+\setminus\overline{B_r}^+$,\label{eq:F(D^2varphi+A)}\\
  $$\varphi(x)=0$$        \ &on\ $\partial\mathbb{R}^n_+$\notag
 \end{numcases}
 and
 $$|\varphi(x)|=O\left(|x|\right),\ |D\varphi(x)|=O\left(1\right)\ \text{and}\ \left|D^2\varphi(x)\right|=O\left(|x|^{-1}\right)\ \text{in}\ \overline{\mathbb{R}^n_+}\setminus\overline{B_r^+},$$ where $r\geq 1$ depends on $n,\lambda,\Lambda,M$ and the $C^{1,1}$ norm of $F$.
 Differentiating equation (\ref{eq:F(D^2varphi+A)}) with respect to $x_k$ for $k=1,2,\cdots,n$, $\varphi_{x_k}(x)$ satisfies $$\bar{a}_{ij}(x)\left(\varphi_{x_k}\right)_{ij}(x)=0\ \text{in}\ \mathbb{R}^n_+\setminus\overline{B_r^+},$$
 where $\bar{a}_{ij}(x)=F_{ij}\left(D^2\varphi(x)+A\right)$ is uniformly elliptic and $$\big|\bar{a}_{ij}(x)-\delta_{ij}\big|=O(|x|^{-1})\ \text{in}\ \overline{\mathbb{R}^n_+}\setminus\overline{B_r^+}.$$

 Noticing $\varphi_{x_k}(x)=0$ on $\partial\mathbb{R}^n_+\cap\{|x^{\prime}|>r\}$ for $k=1,2,\cdots,n-1$, we apply \cite[Theorem 3.4]{JIA-LI-LI} to $\varphi_{x_k}(x)$ to see that
 $$|\varphi_{x_k}(x)|\leq C\frac {x_n}{|x|^n}\ \text{in}\ \overline{\mathbb{R}^n_+}\setminus \overline{B_{R_0}^+},$$ which leads to $$|\varphi_{kn}(x^{\prime},0)|\leq C\frac 1{|x^{\prime}|^n}\ \text{with}\ |x^{\prime}|\geq R_0,$$ where $C$ and $R_0\geq r$ depend on $n,\lambda,\Lambda,M$ and the $C^{1,1}$ norm of $F$.
 Therefore, there exists some constant $b_n$ such that $$\varphi_{x_n}(x^{\prime},0)\to b_n\ \text{as}\ |x^{\prime}|\to\infty$$
 and \cite[Theorem 3.3]{JIA-LI-LI} asserts $$\varphi_{x_n}(x)\to b_n\ \text{as}\ |x|\to\infty.$$ Then $$\varphi(x)\to b_n x_n\ \text{as}\ |x|\to\infty.$$

 Next we estimate the decay rate of $\left|\varphi(x)-b_n x_n\right|$ as $|x|\to\infty$. Let $$\tilde{\varphi}(x)=\varphi(x)-b_n x_n=u(x)-\left(\frac 12x^{\mathrm{T}}Ax+b_n x_n\right).$$
 Obviously, $\tilde{\varphi}(x)$ satisfies equation (\ref{eq:F(D^2varphi+A)}) in $\mathbb{R}^n_+\setminus\overline{B_{R_0}^+}$ and $|D\tilde{\varphi}(x)|\to 0$ as $|x|\to\infty$. Combining with $F(A)=0$, we deduce that $\tilde{\varphi}(x)$ also satisfies the Dirichlet problem
 \begin{equation}\label{eq:tilde a_ij}
  \begin{cases}
   \tilde{a}_{ij}(x)\tilde{\varphi}_{ij}(x)=0\ &\text{in}\ \mathbb{R}^n_+\setminus\overline{B_{R_0}^+},\\
   \tilde{\varphi}(x)=0                      \ &\text{on}\ \partial\mathbb{R}^n_+\cap\{|x^{\prime}|>R_0\},
  \end{cases}
 \end{equation}
 where $\tilde{a}_{ij}(x)=\int_0^1 F_{ij}\left(tD^2\tilde{\varphi}(x)+A\right) \mathrm{d}t$ is uniformly elliptic and $$\left|\tilde{a}_{ij}(x)-\delta_{ij}\right|=O(|x|^{-1})\ \text{in}\ \overline{\mathbb{R}^n_+}\setminus \overline{B_{R_0}^+}.$$
 We employ \cite[Theorem 3.4]{JIA-LI-LI} again to get $$|\tilde{\varphi}(x)|\leq C\frac {x_n}{|x|^n},\ x\in\overline{\mathbb{R}^n_+}\setminus\overline{B_{R_1}^+},$$ where $C$ and $R_1\geq R_0$ depend on $n,\lambda,\Lambda,M$ and the $C^{1,1}$ norm of $F$.

 \textit{Step 2. Finer asymptotic behaviour at infinity.}

 By the decay of $|\tilde{\varphi}(x)|$ estimated in \textit{Step 1} and in view of Lemma \ref{le:higher order decay}, we know that for some $R_2\geq R_1$ depending on $n, \lambda, \Lambda, M$ and the $C^{1,1}$ norm of $F$,
 $$|\tilde{\varphi}(x)|=O\left(\frac {x_n}{|x|^n}\right),\ |D\tilde{\varphi}(x)|=O(|x|^{-n})\ \text{and}\ \left|D^2\tilde{\varphi}(x)\right|=O\left(|x|^{-1-n}\right)\ \text{in}\ \overline{\mathbb{R}^n_+}\setminus \overline{B_{R_2}^+}.$$
 Since $\tilde{\varphi}(x)$ satisfies the Dirichlet problem (\ref{eq:tilde a_ij}) and $$|\tilde{a}_{ij}(x)-\delta_{ij}|=O(|x|^{-1-n})\ \text{in}\ \overline{\mathbb{R}^n_+}\setminus\overline{B_{R_2}^+},$$
 we see that for all $x\in\mathbb{R}^n_+\setminus\overline{B_{R_2}^+}$,
 $$\Delta\tilde{\varphi}(x)=(\delta_{ij}-\tilde{a}_{ij}(x))\tilde{\varphi}_{ij}(x)=:f(x)=O(|x|^{-1-n}|x|^{-1-n})=O(|x|^{-2-2n}).$$

 Let $\psi(x)$ be the Kelvin transform of $\tilde{\varphi}(x)$, that is,
 $$\psi(x)=|x|^{2-n}\tilde{\varphi}\left(\frac x{|x|^2}\right),\ x\in B_{1/R_2}^+.$$
 It is clear $$\psi(x)=O(|x|^{2-n}|x|^{-1+n})=O(|x|),\ x\in B_{1/R_2}^+$$ and that
 \begin{equation*}
  \begin{cases}
   \Delta\psi(x)=|x|^{-2-n}f\left(\frac x{|x|^2}\right)=:\tilde{f}(x)\ &\text{in}\ B_{1/R_2}^+,\\
   \psi(x)=0                    \ &\text{on}\ \partial B_{1/R_2}^+\cap\{x_n=0\}.
  \end{cases}
 \end{equation*}
 Then $\tilde{f}(x)=O\left(|x|^{-2-n}|x|^{2+2n}\right)=O(|x|^n)$ and $\tilde{f}\in C^{\alpha}\left(B_{1/R_2}^+\right)$ for any $\alpha\in(0,1)$, hence $\psi(x)\in C^{2,\alpha}\left(\overline{B_{1/(2R_2)}^+}\right)$.
 So there exists some matrix $\tilde{A}\in\mathbb{R}^{n\times n}, \tilde{b}\in\mathbb{R}^n$ and $\tilde{c}\in\mathbb{R}$ such that for some constant $C>0$,
 $$\left|\psi(x)-\left(\frac 12 x^{\mathrm{T}}\tilde{A}x+\tilde{b}\cdot x+\tilde{c}\right)\right|\leq C|x|^{2+\alpha},\ x\in \overline{B_{1/(2R_2)}^+}.$$
 In view of $\psi(0)=0, \psi_{x_k}(0)=0$ for $k=1,2,\cdots,n-1$ and $\Delta\psi(0)=0$, we deduce $\tilde{c}=0, \tilde{b}=\left(0,\cdots,0,d_n\right)^{\mathrm{T}}$ and $\mathrm{tr}\tilde{A}=0$. We go back to exterior domain to obtain
 $$\left|\tilde{\varphi}(x)-\left(\frac 12 \frac{x^{\mathrm{T}}\tilde{A}x}{|x|^{n+2}}+d_n\frac{x_n}{|x|^n}\right)\right|\leq C|x|^{-\alpha-n},\ x\in\overline{\mathbb{R}^n_+}\setminus\overline{B_{2R_2}^+},$$
 namely, $$u=\frac 12 x^{\mathrm{T}}Ax+b_n x_n+d_n\frac{x_n}{|x|^n}+\frac 12 \frac{x^{\mathrm{T}}\tilde{A}x}{|x|^{n+2}}+O\left(|x|^{-\alpha-n}\right),\ x\in\overline{\mathbb{R}^n_+}\setminus\overline{B_{2R_2}^+}.$$

 \textit{Step 3. Higher order estimates of the error.}

 Furthermore, suppose $F$ is smooth. Let $$\hat{\varphi}(x)=u(x)-\left(\frac 12 x^{\mathrm{T}}Ax+b_n x_n+d_n \frac{x_n}{|x|^n}+\frac 12 \frac{x^{\mathrm{T}}\tilde{A}x}{|x|^{n+2}}\right).$$
 Then the Schauder estimate asserts for all $k\in\mathbb{N}$, $$\left|D^k\hat{\varphi}(x)\right|\leq C(k)|x|^{-\alpha-n-k}.$$
\end{proof}

\begin{proof}[\textit{Proof of Corollary \ref{cor:Liouville of F}}]
 Theorem \ref{thm:main} states that there exists some $A\in\mathcal{S}^n$, $b\in\mathbb{R}^n$ and $c\in\mathbb{R}$ such that
 $$E(x):=u(x)-\left(\frac 12 x^{\mathrm{T}}Ax+b\cdot x+c\right)\to 0\ \text{as}\ |x|\to\infty$$ and $$E(x)=0\ \text{on}\ \partial\mathbb{R}^n_+.$$
 It follows from $F(D^2E(x)+A)=0$ and $F(A)=0$ that $$a_{ij}(x)E_{ij}(x)=0,\  x\in\mathbb{R}^n_+,$$
 where $a_{ij}(x)=\int_0^1F_{ij}(tD^2E(x)+A)\ \mathrm{d}t$ is uniformly elliptic. Then the maximum principle leads to $E(x)\equiv 0$, namely, $$u(x)=\frac 12 x^{\mathrm{T}}Ax+b\cdot x+c\ \text{in}\ \overline{\mathbb{R}^n_+}.$$
\end{proof}

\section{$K$-quasiconformal mappings in half planes and Theorem \ref{thm:main} for dimension 2}

For the case $n=2$ in Theorem \ref{thm:main}, it can be proved by $K$-quasiconformal mappings. Let us recall the definition of $K$-quasiconformal mappings.
\begin{defn}\label{def:q-c}
 A continuously differentiable mapping $w(x)=(p(x),q(x))$ from a domain $\Omega$ in the $x=(x_1,x_2)$ plane to the $w=(p,q)$ plane is $K$-quasiconformal in $\Omega$ if for some constant $K>0$ we have
 $$p^2_1+p^2_2+q^2_1+q^2_2\leq 2K\left(p_1q_2-p_2q_1\right)$$ for all $x\in\Omega$, where $p_i=\frac{\partial}{\partial x_i}p(x), q_i=\frac{\partial}{\partial x_i}q(x), i=1,2$.
\end{defn}
\begin{re}
The domain $\Omega\subset\mathbb{R}^2$ in Definition \ref{def:q-c} can be either bounded or unbounded. If $\Omega=\mathbb{R}^2\setminus\overline{D}$\  ($D\subset\mathbb{R}^2$ is bounded), we say that $w$ is exterior $K$-quasiconformal in $\Omega$.\\
\end{re}

In order to get H\"{o}lder estimate for exterior $K$-quasiconformal mappings in exterior domains in half planes, by Kelvin transform approach developed in \cite{LI-LIU}, it suffices to show that $K$-quasiconformal mappings are H\"{o}lder continuous in half domains. Particularly, the continuity of $K$-quasiconformal mappings at origin can deduce asymptotic behavior of exterior $K$-quasiconformal mappings at infinity.
\begin{thm}\label{thm:Holder of q-c}
 Let $R>0$ and $w=(p,q)$ be $K$-quasiconformal in $B_R^+\subset\mathbb{R}^2_+$ with $K\geq 1$. Assume that $w$ is continuous up to $\left\{|x_1|<R, x_2=0\right\}\setminus\{0\}$, $p=0$ on $\left\{|x_1|<R, x_2=0\right\}\setminus\{0\}$ and $|w|\leq M$ in $B_R^+$.
 Then $w$ is H\"{o}lder continuous in $B_{R^{\prime}}^+\cup\{|x_1|<R^{\prime},x_2=0\}$ for $0<R^{\prime}<R$, and
 \begin{equation*}
  |w(x)-w(y)|\leq C|x-y|^{\alpha},\ x,y\in B_{R^{\prime}}^+\cup\{|x_1|<R^{\prime},x_2=0\},
 \end{equation*}
 where $\alpha=K-(K^2-1)^{\frac 12}$, $C$ depends on $K, R, R^{\prime}$ and $M$.
\end{thm}
\begin{re}
The results in Theorem \ref{thm:Holder of q-c} are also valid for $p,q\in W^{1,2}_{loc}\left(B_R^+\right)\cap L^{\infty}\left(B_R^+\right).$
\end{re}
\begin{proof}[\textit{Proof of Theorem \ref{thm:Holder of q-c}.}]
 We make a reflection across $x_2=0$ so that $$p(x_1,-x_2)=-p(x_1,x_2),\ q(x_1,-x_2)=q(x_1,x_2)$$ in the extended $x$ plane. Then $w$ is $K$-quasiconformal in $B_R\setminus\{0\}$.
 Thanks to the interior H\"{o}lder estimate with interior isolated singularities given by Finn and Serrin \cite[Theorem 3]{F-S}, we have that for $\alpha=K-\left(K^2-1\right)^{\frac12}$ and any $0<R^{\prime}<R$, $$\left|w(x)-w(y)\right|\leq C|x-y|^{\alpha},\ x,y\in B_{R^{\prime}},$$ where $C$ depends on $K, R, R^{\prime}$ and $M$.
 Thus $$\left|w(x)-w(y)\right|\leq C|x-y|^{\alpha},\ x,y\in B_{R^{\prime}}^+\cup \{|x_1|<R^{\prime}, x_2=0\}.$$
\end{proof}
By Kelvin transform, we have H\"{o}lder estimate and asymptotic behavior at infinity for exterior $K$-quasiconformal mappings.
\begin{thm}\label{thm:Holder of ext q-c}
 Let $R>0$ and $w=(p,q)$ be exterior $K$-quasiconformal in $\mathbb{R}^2_+\setminus\overline{B_R^+}$ with $K\geq 1$. Assume that $w$ is continuous up to $\left\{|x_1|>R, x_2=0\right\}$, $q=0$ on $\left\{|x_1|>R, x_2=0\right\}$ and $|w|\leq M$ in $\mathbb{R}^2_+\setminus\overline{B_R^+}$.
 Then $w$ is H\"{o}lder continuous in $\overline{\mathbb{R}^2_+}\setminus\overline{B_{R^{\prime}}^+}$ for $R^{\prime}>R$, namely
 \begin{equation*}
  |w(x)-w(y)|\leq C|x-y|^{\alpha},\ x,y\in \overline{\mathbb{R}^2_+}\setminus\overline{B_{R^{\prime}}^+}
 \end{equation*}
 and $w(x)$ tends to a limit $w(\infty)=\left(p(\infty),0\right)$ at infinity with
 $$\left|w(x)-w(\infty)\right|\leq C|x|^{-\alpha}\ \text{as}\ |x|\to\infty,$$
 where $\alpha=K-(K^2-1)^{\frac 12}$, $C$ depends on $K, R, R^{\prime}$ and $M$.
\end{thm}
\begin{proof}
 Let $$\tilde{p}(x)=p\left(\frac{x}{|x|^2}\right)\ \text{and}\ \ \tilde{q}(x)=q\left(\frac{x}{|x|^2}\right)$$ be Kelvin transform to $p$ and $q$ respectively. Then \cite[Lemma 2.5]{LI-LIU} asserts that $\tilde{w}=\left(\tilde{q},\tilde{p}\right)$ is $K$-quasiconformal in $B_{1/R}^+$.
 By the hypotheses of $w$, we see that $\tilde{w}$ is continuous in $B_{1/R}^+$ up to $\{|x_1|<1/R,x_2=0\}\setminus\{0\}$, $\tilde{q}=0$ on $\{|x_1|<1/R,x_2=0\}\setminus\{0\}$ and $|\tilde{w}|\leq M$ in $B_{1/R}^+$.
 It follows from Theorem \ref{thm:Holder of q-c} that for $R^{\prime}>R$,
 $$\left|\tilde{w}(x)-\tilde{w}(y)\right|\leq C\left|x-y\right|^{\alpha},\ x,y\in B_{1/R^{\prime}}^+\cup\{|x_1|<1/R^{\prime},x_2=0\},$$
 and $\tilde{w}(x)$ has a limit $\tilde{w}(0)=\left(0,\tilde{p}(0)\right)$ at $0$ with $$\left|\tilde{w}(x)-\tilde{w}(0)\right|\leq C\left|x\right|^{\alpha},\ x\in B_{1/R^{\prime}}^+\cup\{|x_1|<1/R^{\prime},x_2=0\},$$
 where $\alpha=K-(K^2-1)^{\frac 12}$, $C$ depends on $K, R, R^{\prime}$ and $M$.

 The theorem follows from returning to exterior domain.
\end{proof}
By virtue of Theorem \ref{thm:Holder of ext q-c}, we can establish the gradient H\"{o}lder estimate and asymptotic behavior at infinity for solutions to linear uniformly elliptic equations with zero boundary value in exterior domains in half planes.
\begin{cor}\label{cor:ext grad est 2d}
 Let $n=2$ and $u\in C^2\left(\overline{\mathbb{R}^2_+}\setminus \overline{B_1^+}\right)$ be a solution to the Dirichlet problem (\ref{eq:linear ext domain}).
 Assume that $u$ satisfies linear growth condition (\ref{cond:linear growth}) in $\mathbb{R}^2_+\setminus\overline{B_1^+}$ for some positive constant $\bar{C}$. Then for $R>1$, there exists some constant $\alpha\in(0,1)$ such that
 $$\left|Du(x)-Du(y)\right|\leq C\left|x-y\right|^{\alpha},\  x,y\in\overline{\mathbb{R}^2_+}\setminus\overline{B_R^+}$$
 and $Du(x)$ has a limit $Du(\infty)=(0, u_{x_2}(\infty))$ at infinity with
 \begin{equation*}\label{decay of gradient}
  |Du(x)-Du(\infty)|\leq C|x|^{-\alpha},\ x\in\overline{\mathbb{R}^2_+}\setminus\overline{B_R^+},
 \end{equation*}
where $\alpha$ depends on $\lambda$ and $\Lambda$, $C$ depends on $\lambda, \Lambda, \bar{C}$ and $R$.
\end{cor}
\begin{re}
The results in Corollary \ref{cor:ext grad est 2d} are also valid for $u\in W^{2,2}\left(\mathbb{R}^2_+\setminus \overline{B_1^+}\right).$
\end{re}
\begin{proof}[\textit{Proof of Corollary \ref{cor:ext grad est 2d}.}]
 Assume without loss of generality that $\lambda=1$. Let $p=u_{x_2},\ q=u_{x_1}$ so that $q=0$ on $\{|x_1|>1,x_2=0\}$. Suppose $\Lambda/\lambda\leq\gamma$ for some $\gamma\geq 1$. Then the uniform ellipticity of $a_{ij}(x)$ and equation in (\ref{eq:linear ext domain}) read (we refer to \cite[Chapter 12]{G-T} for details)
 $$p_1^2+p_2^2+q_1^2+q_2^2\leq (1+\gamma)\left(p_1q_2-p_2q_1\right),\ x\in\mathbb{R}^2_+\setminus\overline{B_1^+}.$$ It is obvious that $w=(p,q)$ is exterior $K$-quasiconformal in $\mathbb{R}^2_+\setminus\overline{B_1^+}$ with $K=(1+\gamma)/2$ and $w$ is continuous up to $\{|x_1|>1,x_2=0\}$. Moreover, $|w|$ is bounded by some constant depending only on $\bar{C}$.
 Then by Theorem \ref{thm:Holder of ext q-c}, we have that for $R>1$,
 $$\left|w(x)-w(y)\right|\leq C\left|x-y\right|^{\alpha},\ x,y\in \overline{\mathbb{R}^2_+}\setminus\overline{B_R^+}$$
 and $w(x)$ has a limit $w(\infty)=(u_{x_2}(\infty),0)$ at infinity with $$\left|w(x)-w(\infty)\right|\leq C\left|x\right|^{-\alpha},\ x\in \overline{\mathbb{R}^2_+}\setminus\overline{B_R^+},$$
 where $\alpha$ depends on $\gamma$, $C$ depends on $\gamma, \bar{C}$ and $R$.
\end{proof}
By Corollary \ref{cor:ext grad est 2d}, we can prove Theorem \ref{thm:main} for case $n=2$. Thanks to the proof of Theorem \ref{thm:main} in Section 3, we only need to determine the asymptotic behavior of the Hessian at infinity.
\begin{thm}\label{thm:limit of Hessian 2d}
 Let $n=2$ and $u$ be as in Theorem \ref{thm:main}. Then there exists a unique $A\in\mathcal{S}^2$ such that
 $$D^2u(x)\to A\ \text{as}\ |x|\to\infty$$ and $$|D^2u(x)-A|\leq C|x|^{-\alpha}\ \text{as}\ |x|\to\infty,$$ which implies
 $$\left|u(x)-\frac 12 x^{\mathrm{T}}Ax\right|\leq C|x|^{2-\alpha}\ \text{as}\ |x|\to\infty$$ and $$\left.\frac 12 x^{\mathrm{T}}Ax\right|_{\partial\mathbb{R}_+^2}=p(x),$$ where $\alpha\in(0,1)$ depends on $\lambda$ and $\Lambda$, $C$ depends on $\lambda,\Lambda$ and $M$.
\end{thm}
\begin{proof}
 According to Nirenberg estimate and Schauder estimate, we see that $u\in C^{3,\tau}\left(\overline{\mathbb{R}^n_+}\setminus \overline{B_1^+}\right)$ for some $\tau\in(0,1)$ depending on $\lambda$ and $\Lambda$.
 As in the proof of Theorem \ref{thm:limit of Hessian}, after proper normalization, $$v(x):=u_{x_1}(x)-x_1$$ satisfies the Dirichlet problem (\ref{eq:linear ext domain}) and it follows from quadratic growth condition (\ref{cond:qua growth of F}) that
 $$|v(x)|\leq \bar{C}\left(1+|x|\right),\ x\in\overline{\mathbb{R}^2_+}\setminus\overline{B_1^+},$$ where $\bar{C}$ depends only on $M$.
 By virtue of Corollary \ref{cor:ext grad est 2d}, we have that $Dv(x)$ tends to a limit $Dv(\infty)=\left(0,v_{x_2}(\infty)\right)$ at infinity and that for some $\alpha\in(0,1)$ depending only on $\lambda$ and $\Lambda$,
 $$\left|Dv(x)-Dv(\infty)\right|\leq C|x|^{-\alpha}\ \text{as}\ |x|\to\infty,$$ where $C$ depends on $\lambda, \Lambda$ and $\bar{C}$.
 Then we can conclude that there exists some constant $u_{12}(\infty)$ such that $$|u_{11}(x)-1|\leq C|x|^{-\alpha}\ \text{and}\ |u_{12}(x)-u_{12}(\infty)|\leq C|x|^{-\alpha}\ \text{as}\ |x|\to\infty.$$
 As for the behavior of $u_{22}$ at infinity, the same way in proof of Theorem \ref{thm:limit of Hessian} asserts that there exists some constant $u_{22}(\infty)$ such that $$|u_{22}(x)-u_{22}(\infty)|\leq C|x|^{-\alpha}\ \text{as}\ |x|\to\infty.$$
 The proof can be closed with
 $A=\begin{pmatrix} 1 & u_{12}(\infty)\\ u_{12}(\infty) & u_{22}(\infty) \end{pmatrix}$.
\end{proof}

\section{applications}
In this section, we apply Theorem \ref{thm:main} to Monge-Amp\`{e}re equations, $k$-Hessian equations and special Lagrangian equations respectively.

For Monge-Amp\`{e}re equations, the asymptotic behavior of convex viscosity solutions at
infinity in exterior domains in half spaces was obtained in \cite{JIA-LI-LI}, which can also be deduced from Pogorelov estimate in half domain in \cite{SA1} and Theorem \ref{thm:main}, i.e. Theorem \ref{thm:M-A eq}.

For $k$-Hessian ($1<k<n$) equations, the Liouville type theorem in exterior domains in half spaces, i.e. Theorem \ref{thm:k-Hessian eq}, can be established by the Pogorelov type estimate in half domains \cite[Theorem 2.1]{JIA-MA} and Theorem \ref{thm:main}. If $k=2$, the results for $n=3$ and $n=4$ can be also deduced from interior Hessian estimates in \cite{Warren-Y1} and \cite{Sh-Y2} respectively.

For special Lagrangian equations with critical and supercritical phases in $\mathbb{R}_+^n\setminus\overline{B_1}^+$, the interior Hessian estimates (cf. \cite{Warren-Y2} for $n=2$ and \cite{Wang-Y} for $n\geq 3$) and quadratic growth condition (\ref{cond:qua growth of s-L}) can deduce a uniform bound of the Hessian in $\mathbb{R}_+^n\setminus\overline{B_R^+}$ for some $R\geq 1$.
Then the convexity of the level set $\left\{\lambda=(\lambda_1,\cdots,\lambda_n)\ |\  \sum_{i=1}^n\arctan\lambda_i=\Theta\right\}$ with $\Theta\geq(n-2)\pi/2$ (cf. \cite[Lemma 2.1]{Y2}) and Evans-Krylov estimate imply the uniform continuity of the Hessian in $\mathbb{R}^n_+\setminus\overline{B_R^+}$. Therefore the Hessian is uniformly bounded in $\overline{\mathbb{R}^n_+}\setminus\overline{B_R^+}$ and the Liouville type theorem, Theorem \ref{thm:s-L eq}, follows from Theorem \ref{thm:main}.

\begin{thm}[\text{\cite[Theorem 1.1]{JIA-LI-LI}}]\label{thm:M-A eq}
 Let $p$ be a quadratic polynomial satisfying $D^2p>0$ and $u\in C\left(\overline{\mathbb{R}^n_+}\setminus \overline{B_1}^+\right)$ be a convex viscosity solution of
 \begin{numcases}{}
  $$\mathrm{det} D^2u(x)=1$$  \ &\text{in}\ $\mathbb{R}^n_+\setminus\overline{B_1}^+$,\label{eq:M-A}\\
  $$u(x)=p(x)$$               \ &\text{on}\ $\partial\mathbb{R}^n_+$.\notag
 \end{numcases}
 Assume that $u$ satisfies
 \begin{equation}
  \mu|x|^2\leq u(x)\leq \mu^{-1}|x|^2\label{cond:qua growth of M-A}
 \end{equation} in $\overline{\mathbb{R}^n_+}\setminus \overline{B_1}^+$ for some $\mu\in\left(0,\frac 12\right]$.
 Then $u\in C^{\infty}\left(\overline{\mathbb{R}^n_+}\setminus \overline{B_R}^+\right)$ for some $R\geq 1$ depending on $n$ and $\mu$ and there exists a unique positive definite $A\in\mathcal{S}^n$ with $\mathrm{det}A=1$, $b\in\mathbb{R}^n, c, d\in \mathbb{R}$ and some matrix $C\in\mathbb{R}^{n\times n}$ such that for any $\alpha\in(0,1)$,
 $$u(x)=\frac 12 x^{\mathrm{T}}Ax+b\cdot x+c+d\frac{x_n}{\left|x^{\mathrm{T}}Ax\right|^{\frac n2}}+\frac 12 \frac{x^{\mathrm{T}}Cx}{\left|x^{\mathrm{T}}Ax\right|^{\frac{n+2}2}}+O_k\left(|x|^{-\alpha-n}\right)\ \text{as}\ |x|\to\infty$$ for all $k\in\mathbb{N}$, where $f(x)=O_k(|x|^n)$ means that $\left|D^kf(x)\right|=O\left(|x|^{n-k}\right)$
 and $$\left.\left(\frac 12 x^{\mathrm{T}}Ax+b\cdot x+c\right)\right|_{\partial\mathbb{R}_+^n}=p(x),\ \left.\left(\frac 12 \frac{x^{\mathrm{T}}Cx}{\left|x^{\mathrm{T}}Ax\right|^{\frac{n+2}2}}+O\left(|x|^{-\alpha-n}\right)\right)\right|_{\partial\mathbb{R}_+^n}=0.$$
\end{thm}

\begin{thm}\label{thm:k-Hessian eq}
 Let $p$ be a $(k-1)$-convex quadratic polynomial for $1<k<n$ and $u\in C^4\left(\overline{\mathbb{R}^n_+}\setminus \overline{B_1}^+\right)$ be a $k$-convex solution to
 \begin{numcases}{}
 $$\sigma_k\left(\lambda(D^2u(x))\right)=1$$\ &in\ $\mathbb{R}^n_+\setminus \overline{B_1}^+$, \label{eq:k-Hessian}\\
 $$u(x)=p(x)$$\          &on\ $\partial\mathbb{R}^n_+.$\notag
 \end{numcases}
 Assume that in $\overline{\mathbb{R}^n_+}\setminus \overline{B_1}^+$, $u$ satisfies (\ref{cond:qua growth of M-A}) for some $\mu>0$ small and $$\sigma_{k+1}\left(\lambda(D^2u(x))\right)\geq -\eta$$ for some $\eta>0$.
 Then $u\in C^{\infty}\left(\overline{\mathbb{R}^n_+}\setminus \overline{B_R}^+\right)$ for some $R\geq 1$ depending on $n, k$ and $\mu$ and there exists a unique $A\in\mathcal{S}^n, b\in\mathbb{R}^n, c, d\in \mathbb{R}$ and some matrix $C\in\mathbb{R}^{n\times n}$ such that for any $\alpha\in(0,1)$,
 $$u(x)=\frac 12 x^{\mathrm{T}}Ax+b\cdot x+c+d\frac{x_n}{\left|x^{\mathrm{T}}Qx\right|^{\frac n2}}+\frac 12 \frac{x^{\mathrm{T}}Cx}{\left|x^{\mathrm{T}}Qx\right|^{\frac{n+2}2}}+O_l\left(|x|^{-\alpha-n}\right)\ \text{as}\ |x|\to\infty$$ for all $l\in\mathbb{N}$ and $$\left.\left(\frac 12 x^{\mathrm{T}}Ax+b\cdot x+c\right)\right|_{\partial\mathbb{R}_+^n}=p(x),\ \left.\left(\frac 12 \frac{x^{\mathrm{T}}Cx}{\left|x^{\mathrm{T}}Qx\right|^{\frac{n+2}2}}+O\left(|x|^{-\alpha-n}\right)\right)\right|_{\partial\mathbb{R}_+^n}=0,$$
 where $Q=\left[\frac{\partial}{\partial M_{ij}}\sigma_k(\lambda(A))\right]^{-1}$.
\end{thm}

\begin{thm}\label{thm:s-L eq}
 Let $(n-2)\pi/2\leq\Theta<n\pi/2$ and $p$ be a quadratic polynomial satisfying $\Theta-\pi/2<\sum_{i=1}^{n-1}\arctan{\lambda_i\left(D^2p(x)\right)}<\Theta+\pi/2$. Let $u$ be a smooth solution to
 \begin{numcases}{}
  $$\sum\limits_{i=1}^n\arctan{\lambda_i\left(D^2u(x)\right)}=\Theta$$\ &in $\mathbb{R}^n_+\setminus\overline{B_1}^+,$\label{eq:s-L}\\
  $$u(x)=p(x)$$\ &on\ $\partial\mathbb{R}^n_+$.\notag
 \end{numcases}
 Assume that for some $\mu>0$,
 \begin{equation}\label{cond:qua growth of s-L}
  |u(x)|\leq \mu\left(1+|x|^2\right)
 \end{equation} in $\overline{\mathbb{R}^n_+}\setminus\overline{B_1}^+$.
 Then there exists a unique $A\in \mathcal{S}^n, b\in\mathbb{R}^n, c,d\in \mathbb{R}$ and some matrix $C\in\mathbb{R}^{n\times n}$ such that for any $\alpha\in(0,1)$,
 $$u(x)=\frac 12 x^{\mathrm{T}}Ax+b\cdot x+c+d\frac{x_n}{\left|x^{\mathrm{T}}Qx\right|^{\frac n2}}+\frac 12 \frac{x^{\mathrm{T}}Cx}{\left|x^{\mathrm{T}}Qx\right|^{\frac{n+2}2}}+O_k\left(|x|^{-\alpha-n}\right)\ \text{as}\ |x|\to\infty$$ for all $k\in\mathbb{N}$
 and $$\left.\left(\frac 12 x^{\mathrm{T}}Ax+b\cdot x+c\right)\right|_{\partial\mathbb{R}_+^n}=p(x),\ \left.\left(\frac 12 \frac{x^{\mathrm{T}}Cx}{\left|x^{\mathrm{T}}Qx\right|^{\frac{n+2}2}}+O\left(|x|^{-\alpha-n}\right)\right)\right|_{\partial\mathbb{R}_+^n}=0,$$ where $Q=I+A^2$.
\end{thm}
The Liouville theorem for special Lagrangian equations in half spaces is a direct corollary of Theorem \ref{thm:s-L eq}.
\begin{cor}\label{cor:s-L eq}
 Let $(n-2)\pi/2\leq\Theta<n\pi/2$ and $p$ satisfy the hypotheses of Theorem \ref{thm:s-L eq}. Let $u$ be a smooth solution to
 \begin{numcases}{}
  $$\sum\limits_{i=1}^n\arctan{\lambda_i\left(D^2u(x)\right)}=\Theta$$\ &in $\mathbb{R}^n_+,$\notag\\
  $$u(x)=p(x)$$\ &on\ $\partial\mathbb{R}^n_+$.\notag
 \end{numcases}
 Assume that $u$ satisfies (\ref{cond:qua growth of s-L}) for some $\mu>0$ in $\overline{\mathbb{R}^n_+}$. Then $u$ must be a quadratic polynomial.
\end{cor}
\begin{re}
 The quadratic growth conditions (\ref{cond:qua growth of M-A}) and (\ref{cond:qua growth of s-L}) in Theorem \ref{thm:M-A eq}-\ref{thm:s-L eq} and Corollary \ref{cor:s-L eq} are reasonable.

 (1) The nonquadratic convex function $$u(x)=\frac{x_1^2}{2(x_n+1)}+\frac 12\left(x_2^2+\cdots+x_{n-1}^2\right)+\frac 16\left(x_n^3+3x_n^2\right)$$ given by Mooney \cite{M} and Savin \cite{SA2} solves equation (\ref{eq:M-A}) in $\mathbb{R}^n_+$ with quadratic boundary data.

 (2) It was proved by Warren \cite{Warren} that for all $n\geq 2k-1$, there exists nonpolynomial $k$-convex elliptic entire solutions to the equation (\ref{eq:k-Hessian}) in $\mathbb{R}^n$.
 Especially, in $\mathbb{R}^3_+$, the function $$u(x_1,x_2,x_3)=\left(x_1^2+x_2^2\right)e^{x_3}+\frac 14e^{-x_3}-e^{x_3}$$ solves equation (\ref{eq:k-Hessian}) with $k=2$, which is equivalent to equation (\ref{eq:s-L}) with critical phase $\Theta=\pi/2$. Obviously, this function is quadratic on $\partial\mathbb{R}^3_+$, but it is not quadratic in $\mathbb{R}^3_+$.
\end{re}

\bibliographystyle{elsarticle-num}

\begin{thebibliography}{99}
\bibitem{B-C-G-J} Bao, Jiguang; Chen, Jingyi; Guan, Bo; Ji, Min. \textit{Liouville property and regularity of a Hessian quotient equation.} Amer. J. Math. 125 (2003), no. 2, 301–316.
\bibitem{C-LI} Caffarelli, L.; Li, Yanyan. \textit{An extension to a theorem of J\"{o}rgens, Calabi, and Pogorelov.} Comm. Pure Appl. Math. 56 (2003), no. 5, 549-583.
\bibitem{CA} Calabi, Eugenio. \textit{Improper affine hyperspheres of convex type and a generalization of a theorem by K. J\"{o}rgens.} Michigan Math. J. 5 (1958), 105-126.
\bibitem{CH-Y} Chang, Sun-Yung Alice; Yuan, Yu. \textit{A Liouville problem for the sigma-2 equation.} Discrete Contin. Dyn. Syst. 28 (2010), no. 2, 659–664.
\bibitem{CH-X} Chen, Li; Xiang, Ni. \textit{Rigidity theorems for the entire solutions of 2-Hessian equation.} J. Differential Equations 267 (2019), no. 9, 5202–5219.
\bibitem{CH-D} Chu, Jianchun; Dinew, S{\l}awomir. \textit{Liouville theorem for a class of Hessian equations.} arXiv:2306.13825.
\bibitem{Du} Du, Shizhong. \textit{Necessary and sufficient conditions to Bernstein theorem of a Hessian equation.} Trans. Amer. Math. Soc. 375 (2022), no. 7, 4873–4892.
\bibitem{F-F-M} Ferrer, L.; Mart\'{i}nez, A.; Mil$\acute{\rm a}$n, F. \textit{An extension of a theorem by K. J\"{o}rgens and a maximum principle at infinity for parabolic affine spheres.} Math. Z. 230 (1999), no. 3, 471-486.
\bibitem{F-S} Finn, Robert; Serrin, James. \textit{On the H\"{o}lder continuity of quasi-conformal and elliptic mappings.} Trans. Amer. Math. Soc. 89 (1958), 1-15.
\bibitem{G-T} Gilbarg, David; Trudinger, Neil S. \textit{Elliptic partial differential equations of second order.} Reprint of the 1998 edition. Classics in Mathematics. Springer-Verlag, Berlin, 2001.
\bibitem{H-M} Han, Qing; Marchenko, Ilya. \textit{Solutions of the Special Lagrangian Equation near Infinity.} arXiv:2501.04254.
\bibitem{H-W} Han, Qing; Wang, Zhehui. \textit{Solutions of the minimal surface equation and of the Monge-Ampère equation near infinity.} J. Reine Angew. Math. 820 (2025), 51–73.
\bibitem{JIA-LI-LI} Jia, Xiaobiao; Li, Dongsheng; Li, Zhisu. \textit{Asymptotic behavior at infinity of solutions of Monge-Amp\`{e}re equations in half spaces.} J. Differential Equations 269 (2020), no. 1, 326-348.
\bibitem{JIA-MA} Jia, Xiaobiao; Ma, Shanshan. \textit{The Liouville theorem for k-Hessian equations in the half space.} Calc. Var. Partial Differential Equations 64, (2025), no. 6, Paper No. 192, 22pp.
\bibitem{JO} J\"{o}rgens, Konrad. \textit{\"{U}ber die L\"{o}sungen der Differentialgleichung $rt-s^2=1$}. Math. Ann. 127 (1954), 130-134.
\bibitem{LI-LI-Y} Li, Dongsheng; Li, Zhisu; Yuan, Yu. \textit{A Bernstein problem for special Lagrangian equations in exterior domains.} Adv. Math. 361 (2020), 106927, 29pp.
\bibitem{LI-LIANG} Li, Dongsheng; Liang, Lichun. \textit{Liouville type theorems of fully nonlinear elliptic equations in half spaces.} Nonlinear Anal. 233 (2023), no. 113292, 6 pp.
\bibitem{LI-LIU} Li, Dongsheng; Liu, Rulin. \textit{Quasiconformal mappings and a Bernstein type theorem over exterior domain in $\mathbb{R}^2$.} Calc. Var. Partial Differential Equations 63, (2024), no. 8, Paper No. 209, 13pp.
\bibitem{LI-R-W} Li, Ming; Ren, Changyu; Wang, Zhizhang. \textit{An interior estimate for convex solutions and a rigidity theorem.} J. Funct. Anal. 270 (2016), no. 7, 2691–2714.
\bibitem{M} Mooney, Connor. \textit{Monge-Amp\`{e}re equation.} https://www.math.uci.edu/$\sim$mooneycr/.
\bibitem{PO} Pogorelov, A. V. \textit{On the improper convex affine hyperspheres.} Geometriae Dedicata 1 (1972), no. 1, 33-46.
\bibitem{SA1} Savin, Ovidiu. \textit{Pointwise $C^{2,\alpha}$  estimates at the boundary for the Monge-Amp\`{e}re equation.} J. Amer. Math. Soc. 26 (2013), no. 1, 63–99.
\bibitem{SA2} Savin, Ovidiu. \textit{A localization theorem and boundary regularity for a class of degenerate Monge-Amp\`{e}re equations.} J. Differential Equations 256 (2014), no. 2, 327–388.
\bibitem{Sh-Y1} Shankar, Ravi; Yuan, Yu. \textit{Rigidity for general semiconvex entire solutions to the sigma-2 equation.} Duke Math. J. 171 (2022), no. 15, 3201–3214.
\bibitem{Sh-Y2} Shankar, Ravi; Yuan, Yu. \textit{Hessian estimates for the sigma-2 equation in dimension four.} Ann. of Math. (2) 201 (2025), no. 2, 489–513.
\bibitem{Wang-Y} Wang, Dake; Yuan, Yu. \textit{Hessian estimates for special Lagrangian equations with critical and supercritical phases in general dimensions.} Amer. J. Math. 136 (2014), no. 2, 481–499.
\bibitem{Warren} Warren, Micah. \textit{Nonpolynomial entire solutions to $\sigma^k$ equations.} Comm. Partial Differential Equations 41 (2016), no. 5, 848–853.
\bibitem{Warren-Y1} Warren, Micah; Yuan, Yu. \textit{Hessian estimates for the sigma-2 equation in dimension 3.} Comm. Pure Appl. Math. 62 (2009), no. 3, 305–321.
\bibitem{Warren-Y2} Warren, Micah; Yuan, Yu. \textit{Explicit gradient estimates for minimal Lagrangian surfaces of dimension two.} Math. Z. 262 (2009), no. 4, 867–879.
\bibitem{Y1} Yuan, Yu. \textit{A Bernstein problem for special Lagrangian equations.} Invent. Math. 150 (2002), no. 1, 117–125.
\bibitem{Y2} Yuan, Yu. \textit{Global solutions to special Lagrangian equations.} Proc. Amer. Math. Soc. 134 (2006), no. 5, 1355–1358.
\bibitem{ZH} Zhou, Ziwei. \textit{A Liouville theorem of the 2-Hessian equation in half-space.} J. Math. Anal. Appl. 528 (2023), no. 2, Paper No. 127563, 16 pp.
\end{thebibliography}

\end{document}